\documentclass{amsart}
\usepackage{cases}
\usepackage{latexsym}
\usepackage{amsmath}
\usepackage[arrow,matrix]{xy}
\usepackage{stmaryrd}
\usepackage{amsfonts}
\usepackage{amsmath,amssymb,amscd,bbm,amsthm,mathrsfs,dsfont}

\usepackage{fancyhdr}
\usepackage{amsxtra,ifthen}
\usepackage{verbatim}

\numberwithin{equation}{section}

\theoremstyle{plain}
\newtheorem{theorem}{Theorem}[section]
\newtheorem{lemma}[theorem]{Lemma}
\newtheorem{proposition}[theorem]{Proposition}
\newtheorem{corollary}[theorem]{Corollary}

\theoremstyle{definition}

\newtheorem{remark}[theorem]{Remark}

\begin{document}

\title[Jordan Derivations of Incidence Algebras]
{Jordan Derivations of Incidence Algebras}

\author{Zhankui Xiao}

\address{Xiao: School of Mathematical Sciences, Huaqiao University,
Quanzhou, Fujian, 362021, P. R. China}

\email{zhkxiao@gmail.com}

\begin{abstract}
Let $\mathcal{R}$ be a commutative ring with identity, $I(X,\mathcal{R})$ be the incidence
algebra of a locally finite pre-ordered set $X$. In this note, we characterise the derivations
of $I(X,\mathcal{R})$ and prove that every Jordan derivation of $I(X,\mathcal{R})$ is a derivation
provided that $\mathcal{R}$ is $2$-torsion free.
\end{abstract}



\thanks{The work was supported by a research foundation of
NSFC. (Grant No. 11226068).}

\maketitle

\section{Introduction}\label{xxsec1}

Let $\mathcal{R}$ be a commutative ring with identity, $A$ be an algebra over $\mathcal{R}$.
An $\mathcal{R}$-linear mapping $D: A\rightarrow A$ is called a derivation if $D(xy)=D(x)y+xD(y)$
for all $x,y\in A$, and is called a {\em Jordan derivation} if
$$
D(x^2)=D(x)x+xD(x)
$$
for all $x\in A$. There has been a great interest in the study of Jordan derivations of various algebras
in the last decades. The standard problem is to find out whether a Jordan derivation degenerate
to a derivation. Jacobson and Rickart \cite{JaRi} proved that every Jordan derivation of the full matrix algebra
over a $2$-torsion free unital ring is a derivation by relating the problem to the decomposition of Jordan
homomorphisms. In \cite{He}, Herstein showed that every Jordan derivation from a $2$-torsion free prime ring
into itself is also a derivation. These results have been extended to different rings and algebras in various
directions (see \cite{Br1,Br2,Cu,Lu,ZY} and the references therein). We would like to refer the reader to
Bre\v{s}ar's paper \cite{Br3} for a comprehensive and more detailed understanding of this topic.

We now recall the definition of incidence algebras. Let $(X,\leqslant)$ be a locally finite pre-ordered
set. This means $\leqslant$ is a reflexive and transitive binary relation on the set $X$, and for any
$x\leqslant y$ in $X$ there are only finitely many elements $z$ satisfying $x\leqslant z\leqslant y$.
The {\em incidence algebra} $I(X,\mathcal{R})$ of $X$ over $\mathcal{R}$ is defined on the set
$$
I(X,\mathcal{R}):=\{f: X\times X\longrightarrow \mathcal{R}\mid f(x,y)=0\ \text{if}\ x\nleqslant y\}
$$
with algebraic operation given by
$$\begin{aligned}
(f+g)(x,y)&=f(x,y)+g(x,y),\\
(rf)(x,y)&=rf(x,y),\\
(fg)(x,y)&=\sum_{x\leqslant z\leqslant y}f(x,z)g(z,y)
\end{aligned}$$
for all $f,g\in I(X,\mathcal{R})$, $r\in \mathcal{R}$ and $x,y,z\in X$. The product $(fg)$ is usually called
{\em convolution} in function theory. It would be helpful to point out that the full matrix algebra ${\rm M}_n(\mathcal{R})$
and the upper (or lower) triangular matrix algebras ${\rm T}_n(\mathcal{R})$ are special examples of incidence algebras

Incidence algebras were first considered by Ward \cite{Wa} as generalized algebras of arithmetic functions.
Rota and Stanley developed incidence algebras as the fundamental structures of enumerative combinatorial
theory and allied areas of arithmetic function theory (see \cite{SpDo}). Motivated by the results
of Stanley \cite{St}, automorphisms and other algebraic mappings of incidence algebras have been
extensively studied (see \cite{Ba,CoM,Jo,Ko,Ma,No1,No2,Sp} and the references therein). Baclawski
\cite{Ba} studied the automorphisms and derivations of incidence algebras $I(X,\mathcal{R})$ when $X$ is a locally
finite partially ordered set. Especially, he proved that every derivation of $I(X,\mathcal{R})$ with $X$ a locally
finite partially ordered set can be decomposed as a sum of an inner derivation and a transitive induced derivation
(see the definition in Section \ref{xxsec2}). Koppinen \cite{Ko} has extended these results to the incidence algebras
$I(X,\mathcal{R})$ with $X$ a locally finite pre-ordered set. In present paper, we first characterise the derivations
of $I(X,\mathcal{R})$ by a direct computation. Based on such characterization of derivations, we prove that every
Jordan derivation of $I(X,\mathcal{R})$ is a derivation
provided that $\mathcal{R}$ is $2$-torsion free.

\section{Derivations on Incidence Algebras}\label{xxsec2}

In this section, we characterise the derivations of the incidence algebras $I(X,\mathcal{R})$ where
$X$ is a locally finite pre-ordered set. We first single out some properties of $I(X,\mathcal{R})$ for later
use. One can find out all information in \cite{SpDo} and \cite{Ko}.

The identity element $\delta$ of $I(X,\mathcal{R})$ is given by $\delta(x,y)=\delta_{xy}$ for $x\leqslant y$,
where $\delta_{xy}\in \{0,1\}$ is the Kronecker delta. If $x,y\in X$ with $x\leqslant y$, let $e_{xy}$ be
defined by $e_{xy}(u,v)=1$ if $(u,v)=(x,y)$, and $e_{xy}(u,v)=0$ otherwise. Then $e_{xy}e_{uv}=\delta_{yu}e_{xv}$
by the definition of convolution. Moreover, the set $\mathfrak{B}:=\{e_{xy}\mid x\leqslant y\}$ forms an $\mathcal{R}$-linear
basis of $I(X,\mathcal{R})$. The following lemma, to some extent, is well-known and we give out the proof here
for completeness.

\begin{lemma}\label{xx2.1}
Let $A$ be an $\mathcal{R}$-algebra with an $\mathcal{R}$-linear basis $Y$. Then an $\mathcal{R}$-linear operator
$D: A\rightarrow A$ is a derivation if and only if
$$
D(xy)=D(x)y+xD(y), \hspace{6pt}\forall x,y\in Y.
$$
\end{lemma}

\begin{proof}
We only need to prove the sufficiency. Let $a,b\in A$ and
$$
a=\sum_{x\in Y}C_xx, \hspace{6pt} b=\sum_{y\in Y}C'_yy,
$$
where $C_x,C'_y\in\mathcal{R}$. Then
$$\begin{aligned}
D(ab)&=D\big(\sum_{x,y\in Y}C_xC'_y xy \big)\\
&=\sum_{x,y\in Y}C_xC'_y(D(x)y+xD(y))\\
&=D\big(\sum_{x\in Y}C_xx\big)\big(\sum_{y\in Y}C'_yy\big)+\big(\sum_{x\in Y}C_xx\big)D\big(\sum_{y\in Y}C'_yy\big)\\
&=D(a)b+aD(b).
\end{aligned}$$
\end{proof}

Let $D: I(X,\mathcal{R})\rightarrow I(X,\mathcal{R})$ be an $\mathcal{R}$-linear operator. Motivated by the
above Lemma, we denote for all $i,j\in X$ with $i\leqslant j$
$$
D(e_{ij})=\sum_{e_{xy}\in \mathfrak{B}}C^{ij}_{xy}e_{xy}.
$$
We will use this notation for the rest of this paper.
For any $x\in X$, define $L_x$ and $R_x$ to be the subsets of $X$ by
$$
L_x:=\{i\in X\mid i\leqslant x, i\neq x\}\hspace{6pt}\text{and}\hspace{6pt} R_x:=\{j\in X\mid x\leqslant j, j\neq x\}.
$$
Note that the set $L_x\cap R_x$ maybe not empty since $X$ is a pre-ordered set.

\begin{theorem}\label{xx2.2}
Let $D: I(X,\mathcal{R})\rightarrow I(X,\mathcal{R})$ be an $\mathcal{R}$-linear operator. Then $D$ is a derivation
if and only if $D$ satisfies
$$
D(e_{ij})=\sum_{x\in {L_i}}C^{ii}_{xi}e_{xj}+C^{ij}_{ij}e_{ij}+\sum_{y\in {R_j}}C^{jj}_{jy}e_{iy} \eqno(1)
$$
for all $e_{ij}\in \mathfrak{B}$, where the coefficients $C^{ij}_{xy}$ are subject to the following relations
$$\begin{aligned}
C^{jj}_{jk}+C^{kk}_{jk}&=0, \hspace{6pt} \text{if}\hspace{6pt} j\leqslant k;\\
C^{ij}_{ij}+C^{jk}_{jk}&=C^{ik}_{ik}, \hspace{6pt} \text{if}\hspace{6pt} i\leqslant j,j\leqslant k.
\end{aligned}$$
\end{theorem}

\begin{proof}
($\Longrightarrow$). For any $i\in X$, there is $D(e_{ii})=D(e_{ii})e_{ii}+e_{ii}D(e_{ii})$. Thus
$$\begin{aligned}
\sum_{e_{xy}\in \mathfrak{B}}C^{ii}_{xy}e_{xy}&=\big(\sum_{e_{xy}\in \mathfrak{B}}C^{ii}_{xy}e_{xy} \big)e_{ii}
+e_{ii}\big(\sum_{e_{xy}\in \mathfrak{B}}C^{ii}_{xy}e_{xy} \big)\\
&=\sum_{x,x\leqslant i}C^{ii}_{xi}e_{xi}+\sum_{y,i\leqslant y}C^{ii}_{iy}e_{iy}.
\end{aligned}\eqno(2)$$
Taking $x=i$ and $y=i$ in $(2)$, we have $C^{ii}_{ii}=0$. Hence
$$
D(e_{ii})=\sum_{x\in {L_i}}C^{ii}_{xi}e_{xi}+\sum_{y\in {R_i}}C^{ii}_{iy}e_{iy}. \eqno(3)
$$
For any $e_{ij}\in \mathfrak{B}$ with $i\neq j$, we deduce from equation $(3)$ that
$$\begin{aligned}
D(e_{ij})&=D(e_{ii}e_{ij}e_{jj})\\
&=D(e_{ii})e_{ij}+e_{ii}D(e_{ij})e_{jj}+e_{ij}D(e_{jj})\\
&=D(e_{ii})e_{ij}+C^{ij}_{ij}e_{ij}+e_{ij}D(e_{jj})\\
&=\sum_{x\in {L_i}}C^{ii}_{xi}e_{xj}+C^{ij}_{ij}e_{ij}+\sum_{y\in {R_j}}C^{jj}_{jy}e_{iy}.
\end{aligned}\eqno(4)$$
Note that the identities $(3)$ and $(4)$ proves $(1)$.

In order to determine the coefficients $C^{ij}_{xy}$, we need to apply the derivation $D$ to the identity
$e_{ij}e_{kl}=\delta_{jk}e_{il}$. There are two subcases.

{\bf Case 1.} When $j\neq k$, then
$$\begin{aligned}
0&=D(e_{ij})e_{kl}+e_{ij}D(e_{kl})\\
&=\Big(\sum_{x\in {L_i}}C^{ii}_{xi}e_{xj}+C^{ij}_{ij}e_{ij}+\sum_{y\in {R_j}}C^{jj}_{jy}e_{iy}  \Big)e_{kl}\\
&\quad +e_{ij}\Big(\sum_{x\in {L_k}}C^{kk}_{xk}e_{xl}+C^{kl}_{kl}e_{kl}+\sum_{y\in {R_l}}C^{ll}_{ly}e_{ky}  \Big)\\
&=C^{jj}_{jk}e_{il}+C^{kk}_{jk}e_{il}.
\end{aligned}\eqno(5)$$
Therefore $C^{jj}_{jk}+C^{kk}_{jk}=0$.

{\bf Case 2.} When $j=k$, then
$$\begin{aligned}
D(e_{il})&=D(e_{ij})e_{jl}+e_{ij}D(e_{jl})\\
&=\Big(\sum_{x\in {L_i}}C^{ii}_{xi}e_{xj}+C^{ij}_{ij}e_{ij}+\sum_{y\in {R_j}}C^{jj}_{jy}e_{iy}  \Big)e_{jl}\\
&\quad +e_{ij}\Big(\sum_{x\in {L_j}}C^{jj}_{xj}e_{xl}+C^{jl}_{jl}e_{jl}+\sum_{y\in {R_l}}C^{ll}_{ly}e_{jy}  \Big)\\
&=\sum_{x\in {L_i}}C^{ii}_{xi}e_{xl}+C^{ij}_{ij}e_{il}+C^{jl}_{jl}e_{il}+\sum_{y\in {R_l}}C^{ll}_{ly}e_{iy}.
\end{aligned}\eqno(6)$$
Taking $(1)$ into account, we obtain $C^{ij}_{ij}+C^{jl}_{jl}=C^{il}_{il}$, which proves the necessity.

($\Longleftarrow$). First it follows from the equation $C^{ij}_{ij}+C^{jk}_{jk}=C^{ik}_{ik}$ that $C^{ii}_{ii}=0$
for all $i\in X$. Hence by $(1)$, $D(e_{ii})$ is of the form as in equation $(3)$, i.e.
$$
D(e_{ii})=\sum_{x\in {L_i}}C^{ii}_{xi}e_{xi}+\sum_{y\in {R_i}}C^{ii}_{iy}e_{iy}. \eqno(7)
$$
Now by Lemma \ref{xx2.1}, to show that $D$ is a derivation we just need to show that
$D(e_{ij}e_{kl})=D(e_{ij})e_{kl}+e_{ij}D(e_{kl})$ for all $e_{ij},e_{kl}\in\mathfrak{B}$.
On the other hand, a little modification of the identities $(5)$ and $(6)$ shows that
the equations $D(e_{ij}e_{kl})=D(e_{ij})e_{kl}+e_{ij}D(e_{kl})$ for all $e_{ij},e_{kl}\in\mathfrak{B}$ always
hold. This completes the proof of the theorem.
\end{proof}

For any $g\in I(X,\mathcal{R})$, the mapping $f\mapsto [g,f]$ is a derivation of $I(X,\mathcal{R})$.
We call it an {\em inner derivation} and denote it as ${\rm Inn}_g$. A mapping $f: \leqslant\longrightarrow
\mathcal{R}$ is called {\em transitive} if
$$
f(i,j)+f(j,k)=f(i,k)
$$
for all $i,j,k\in X$ such that $i\leqslant j, j\leqslant k$. Note that every mapping $\sigma: X\longrightarrow \mathcal{R}$
determines a transitive mapping $(i,j)\mapsto \sigma(i)-\sigma(j)$. Transitive mappings of this form are called
{\em trivial}. If $f: \leqslant\longrightarrow \mathcal{R}$ is a transitive mapping, we define an $\mathcal{R}$-linear mapping
$\Delta_f: I(X,\mathcal{R})\longrightarrow I(X,\mathcal{R})$ by
$$
\Delta_f(e_{ij})=f(i.j)e_{ij}
$$
for all $e_{ij}\in \mathfrak{B}$.

\begin{lemma}\label{xx2.3}
With notations as above, $\Delta_f$ is a derivation. Moreover $\Delta_f$ is inner if and only if $f$ is trivial.
\end{lemma}

\begin{proof}
Note that $\Delta_f(e_{ij})e_{kl}+e_{ij}\Delta_f(e_{kl})=\delta_{jk}(f(i,j)+f(k,l))e_{il}=\Delta_f(e_{ij}e_{kl})$ for
all $e_{ij},e_{kl}\in\mathfrak{B}$. Then Lemma \ref{xx2.1} implies that $\Delta_f$ is a derivation.

If $f$ is trivial, there exists a mapping $\sigma: X\longrightarrow \mathcal{R}$ such that $f(i,j)=\sigma(i)-\sigma(j)$.
Let $g=\sum_{i\in X}\sigma(i)e_{ii}\in I(X,\mathcal{R})$. Then $[g,e_{ij}]=(\sigma(i)-\sigma(j))e_{ij}=\Delta_f(e_{ij})$
for all $e_{ij}\in\mathfrak{B}$, and hence $\Delta_f$ is inner. Conversely, if $\Delta_f$ is inner, there exists an
element $g=\sum_{e_{ij}\in\mathfrak{B}}C_{ij}e_{ij}$ such that $\Delta_f(e_{ij})=[g,e_{ij}]$ for all $e_{ij}\in\mathfrak{B}$.
Hence $f(i,j)=C_{ii}-C_{jj}$. Setting $\sigma: X\longrightarrow \mathcal{R}$ by $\sigma(i)=C_{ii}$, we get that $f$ is trivial.
\end{proof}

\begin{remark}\label{xx2.4}
We will call the derivation $\Delta_f$ by {\em transitive induced derivation} induced by the transitive mapping $f$.
The definition of transitive mappings first appeared in Nowicki's paper \cite{No1}, where he dealt with finite incidence algebras, i.e.
$X$ is a finite pre-ordered set.
\end{remark}

\begin{proposition}\label{xx2.5}
Let $D$ be a derivation of $I(X,\mathcal{R})$. Then there exist an element $g\in I(X,\mathcal{R})$ and a transitive mapping
$f$ such that
$$
D={\rm Inn}_g+\Delta_f.
$$
\end{proposition}

\begin{proof}
It follows from Theorem \ref{xx2.2} that
$$
D(e_{ij})=\sum_{x\in {L_i}}C^{ii}_{xi}e_{xj}+C^{ij}_{ij}e_{ij}+\sum_{y\in {R_j}}C^{jj}_{jy}e_{iy}
$$
for all $e_{ij}\in \mathfrak{B}$, where the coefficients $C^{ij}_{xy}$ are subject to the following relations
$$
C^{jj}_{jk}+C^{kk}_{jk}=0\hspace{8pt}\text{and}\hspace{8pt} C^{ij}_{ij}+C^{jk}_{jk}=C^{ik}_{ik}.
$$
Define $f: \leqslant\longrightarrow \mathcal{R}$ by $(i,j)\mapsto C^{ij}_{ij}$. Then $f$ is transitive since
$C^{ij}_{ij}+C^{jk}_{jk}=C^{ik}_{ik}$. On the other hand, Let $g=\sum_{e_{ij}\in\mathfrak{B}}C^{jj}_{ij}e_{ij}$.
Note that $g(i,j)=C^{jj}_{ij}=-C^{ii}_{ij}$. Hence
$$
D(e_{ij})={\rm Inn}_g(e_{ij})+\Delta_f(e_{ij})
$$
for all $e_{ij}\in\mathfrak{B}$. The proposition follows from the fact that the mappings $D$, ${\rm Inn}_g$ and $\Delta_f$
are all $\mathcal{R}$-linear.
\end{proof}

We point out that Proposition \ref{xx2.5} can also be obtained from \cite[Theorem 6.3]{Ko}, where Koppinen
use deeply the coalgebra structure of incidence algebras. Our approach is elementary but is convenient for
studying Jordan derivations in next section. Lemma \ref{xx2.3} and Proposition \ref{xx2.5} deduce the following.

\begin{corollary}\label{xx2.6}
Every derivation of $I(X,\mathcal{R})$ is inner if and only if every transitive mapping is trivial.
\end{corollary}

\begin{corollary}\label{xx2.7}
The incidence algebra $I(X,\mathcal{R})$ is a Kadison algebra, i.e. every local derivation of $I(X,\mathcal{R})$
is a derivation.
\end{corollary}

\begin{proof}
The proof is similar to that of \cite[Theorem 3]{No2}.
\end{proof}

\section{Jordan Derivations on Incidence Algebras}\label{xxsec3}

Throughout this section, we assume that $\mathcal{R}$ is $2$-torsion free. Let
$D: I(X,\mathcal{R})\longrightarrow I(X,\mathcal{R})$ be a Jordan derivation. We also denote
$$
D(e_{ij})=\sum_{e_{xy}\in \mathfrak{B}}C^{ij}_{xy}e_{xy}.
$$
for all $i,j\in X$ with $i\leqslant j$. We assume $C^{ij}_{xy}=0$, if needed, for $x\nleqslant y$.
The following lemma is due to Herstein \cite{He}.

\begin{lemma}\label{xx3.1}
Let $A$ be a $2$-torsion free ring, $D: A\longrightarrow A$ be a Jordan derivation. Then for any
$a,b,c\in A$, we have
\begin{enumerate}
\item[(1)] $D(ab+ba)=D(a)b+aD(b)+D(b)a+bD(a)$,
\item[(2)] $D(aba)=D(a)ba+aD(b)a+abD(a)$,
\item[(3)] $D(abc+cba)=D(a)bc+aD(b)c+abD(c)+D(c)ba+cD(b)a+cbD(a)$.
\end{enumerate}
\end{lemma}

\begin{lemma}\label{xx3.2}
Let $D: I(X,\mathcal{R})\rightarrow I(X,\mathcal{R})$ be a Jordan derivation. Then
$$
D(e_{ij})=\sum_{x\in {L_i}}C^{ii}_{xi}e_{xj}+C^{ij}_{ij}e_{ij}+\sum_{y\in {R_j}}C^{jj}_{jy}e_{iy}+C^{ij}_{ji}e_{ji}
$$
for all $e_{ij}\in \mathfrak{B}$, where the coefficients $C^{ij}_{xy}$ are subject to the following relations
$$\begin{aligned}
C^{jj}_{jk}+C^{kk}_{jk}&=0, \hspace{6pt} \text{if}\hspace{6pt} j\leqslant k;\\
C^{ij}_{ij}+C^{jk}_{jk}&=C^{ik}_{ik}, \hspace{6pt} \text{if}\hspace{6pt} i\leqslant j,j\leqslant k.
\end{aligned}$$
\end{lemma}

\begin{proof}
Since $e_{ii}$ is an idempotent element for all $i\in X$, a similar computation as in Theorem \ref{xx2.2} shows that
$C^{ii}_{ii}=0$ and
$$
D(e_{ii})=\sum_{x\in {L_i}}C^{ii}_{xi}e_{xi}+\sum_{y\in {R_i}}C^{ii}_{iy}e_{iy}. \eqno(8)
$$
For any $e_{ij}\in \mathfrak{B}$ with $i\neq j$, we deduce from Lemma \ref{xx3.1} $(3)$ and equation $(8)$ that
$$\begin{aligned}
D(e_{ij})&=D(e_{ii}e_{ij}e_{jj}+e_{jj}e_{ij}e_{ii})\\
&=D(e_{ii})e_{ij}+e_{ii}D(e_{ij})e_{jj}+e_{ij}D(e_{jj})+e_{jj}D(e_{ij})e_{ii}\\
&=D(e_{ii})e_{ij}+C^{ij}_{ij}e_{ij}+e_{ij}D(e_{jj})+C^{ij}_{ji}e_{ji}\\
&=\sum_{x\in {L_i}}C^{ii}_{xi}e_{xj}+C^{ij}_{ij}e_{ij}+\sum_{y\in {R_j}}C^{jj}_{jy}e_{iy}+C^{ij}_{ji}e_{ji}.
\end{aligned}\eqno(9)$$
Combining the equations $(8)$ and $(9)$, we get the form of $D(e_{ij})$. It would be helpful to point out that
$C^{ij}_{ji}=0$ if $j\nleqslant i$ by our assumption.

In order to determine the coefficients $C^{ij}_{xy}$, we need to apply the Jordan derivation $D$ to the identity
$e_{ij}e_{kl}+e_{kl}e_{ij}=\delta_{jk}e_{il}+\delta_{li}e_{kj}$. There are four subcases.

{\bf Case 1.} When $l\neq i$ and $j\neq k$, we have
$$\begin{aligned}
0&=D(e_{ij})e_{kl}+e_{ij}D(e_{kl})+D(e_{kl})e_{ij}+e_{kl}D(e_{ij})\\
&=\Big(\sum_{x\in {L_i}}C^{ii}_{xi}e_{xj}+C^{ij}_{ij}e_{ij}+\sum_{y\in {R_j}}C^{jj}_{jy}e_{iy}+C^{ij}_{ji}e_{ji}  \Big)e_{kl}\\
&\quad +e_{ij}\Big(\sum_{x\in {L_k}}C^{kk}_{xk}e_{xl}+C^{kl}_{kl}e_{kl}+\sum_{y\in {R_l}}C^{ll}_{ly}e_{ky}+C^{kl}_{lk}e_{lk}  \Big)\\
&\quad +\Big(\sum_{x\in {L_k}}C^{kk}_{xk}e_{xl}+C^{kl}_{kl}e_{kl}+\sum_{y\in {R_l}}C^{ll}_{ly}e_{ky}+C^{kl}_{lk}e_{lk}  \Big)e_{ij}\\
&\quad +e_{kl}\Big(\sum_{x\in {L_i}}C^{ii}_{xi}e_{xj}+C^{ij}_{ij}e_{ij}+\sum_{y\in {R_j}}C^{jj}_{jy}e_{iy}+C^{ij}_{ji}e_{ji}  \Big)\\
&=C^{jj}_{jk}e_{il}+\delta_{ik}C^{ij}_{ji}e_{jl}+C^{kk}_{jk}e_{il}+\delta_{jl}C^{kl}_{lk}e_{ik}\\
&\quad +C^{ll}_{li}e_{kj}+\delta_{ik}C^{kl}_{lk}e_{lj}+C^{ii}_{li}e_{kj}+\delta_{jl}C^{ij}_{ji}e_{ki}.
\end{aligned}\eqno(10)$$
We consider the coefficient of $e_{il}$ in the equation $(10)$. If $e_{il}=e_{jl}$, then $C^{ij}_{ji}=C^{ii}_{ii}=0$.
If $e_{il}=e_{ik}$, then $C^{kl}_{lk}=C^{kk}_{kk}=0$. Therefore, $C^{jj}_{jk}+C^{kk}_{jk}=0$ if $(i,l)\neq (k,j)$ or
$2(C^{jj}_{jk}+C^{kk}_{jk})=0$ if $(i,l)=(k,j)$. Hence $C^{jj}_{jk}+C^{kk}_{jk}=0$ for $j\leqslant k$ since
$\mathcal{R}$ is $2$-torsion free.

{\bf Case 2.} When $l\neq i$ and $j=k$, we have
$$\begin{aligned}
D(e_{il})&=D(e_{ij})e_{kl}+e_{ij}D(e_{kl})+D(e_{kl})e_{ij}+e_{kl}D(e_{ij})\\
&=\Big(\sum_{x\in {L_i}}C^{ii}_{xi}e_{xj}+C^{ij}_{ij}e_{ij}+\sum_{y\in {R_j}}C^{jj}_{jy}e_{iy}+C^{ij}_{ji}e_{ji}  \Big)e_{jl}\\
&\quad +e_{ij}\Big(\sum_{x\in {L_j}}C^{jj}_{xj}e_{xl}+C^{jl}_{jl}e_{jl}+\sum_{y\in {R_l}}C^{ll}_{ly}e_{jy}+C^{jl}_{lj}e_{lj}  \Big)\\
&\quad +\Big(\sum_{x\in {L_j}}C^{jj}_{xj}e_{xl}+C^{jl}_{jl}e_{jl}+\sum_{y\in {R_l}}C^{ll}_{ly}e_{jy}+C^{jl}_{lj}e_{lj}  \Big)e_{ij}\\
&\quad +e_{jl}\Big(\sum_{x\in {L_i}}C^{ii}_{xi}e_{xj}+C^{ij}_{ij}e_{ij}+\sum_{y\in {R_j}}C^{jj}_{jy}e_{iy}+C^{ij}_{ji}e_{ji}  \Big)\\
&=\sum_{x\in {L_i}}C^{ii}_{xi}e_{xl}+C^{ij}_{ij}e_{il}+\delta_{ij}C^{ij}_{ji}e_{jl}
+C^{jl}_{jl}e_{il}+\sum_{y\in {R_l}}C^{ll}_{ly}e_{iy}+\delta_{jl}C^{jl}_{lj}e_{ij}\\
&\quad +C^{ll}_{li}e_{jj}+\delta_{ij}C^{jl}_{lj}e_{lj}+C^{ii}_{li}e_{jj}+\delta_{lj}C^{ij}_{ji}e_{ji}\\
&=\sum_{x\in {L_i}}C^{ii}_{xi}e_{xl}+C^{ij}_{ij}e_{il}+C^{jl}_{jl}e_{il}+\sum_{y\in {R_l}}C^{ll}_{ly}e_{iy}+\delta_{ij}C^{jl}_{lj}e_{lj}+\delta_{lj}C^{ij}_{ji}e_{ji},
\end{aligned}\eqno(11)$$
where the last identity follows from the fact $\delta_{ij}C^{ij}_{ji}=0$ and $C^{ll}_{li}+C^{ii}_{li}=0$.
Therefore, there is
$$
C^{il}_{il}e_{il}+C^{il}_{li}e_{li}=C^{ij}_{ij}e_{il}+C^{jl}_{jl}e_{il}+\delta_{ij}C^{jl}_{lj}e_{lj}+\delta_{lj}C^{ij}_{ji}e_{ji}.\eqno(12)
$$
We consider the coefficient of $e_{il}$ in the equation $(12)$. Note that $l\neq i$ in this case. We obtain
$C^{ij}_{ij}+C^{jl}_{jl}=C^{il}_{il}$.

{\bf Case 3.} When $l=i$ and $j\neq k$, this case is same with the Case 2.

{\bf Case 4.} When $l=i$ and $j=k$, a direct calculation shows that $C^{ij}_{ij}+C^{ji}_{ji}=0$. Combining with the Case 2 or the
Case 3, we have $C^{ij}_{ij}+C^{jl}_{jl}=C^{il}_{il}$ for all $i\leqslant j$ and $j\leqslant l$.
\end{proof}

We now are at the position to prove the main result of this paper.

\begin{theorem}\label{xx3.3}
Let $\mathcal{R}$ be a $2$-torsion free commutative ring with identity. Then every Jordan derivation of the incidence
algebra $I(X,\mathcal{R})$ is a derivation.
\end{theorem}

\begin{proof}
Let $D: I(X,\mathcal{R})\longrightarrow I(X,\mathcal{R})$ be a Jordan derivation. By Lemma \ref{xx3.2}
$$
D(e_{ij})=\sum_{x\in {L_i}}C^{ii}_{xi}e_{xj}+C^{ij}_{ij}e_{ij}+\sum_{y\in {R_j}}C^{jj}_{jy}e_{iy}+C^{ij}_{ji}e_{ji}
$$
for all $e_{ij}\in \mathfrak{B}$, where the coefficients $C^{ij}_{xy}$ are subject to the following relations
$$\begin{aligned}
C^{jj}_{jk}+C^{kk}_{jk}&=0, \hspace{6pt} \text{if}\hspace{6pt} j\leqslant k;\\
C^{ij}_{ij}+C^{jk}_{jk}&=C^{ik}_{ik}, \hspace{6pt} \text{if}\hspace{6pt} i\leqslant j,j\leqslant k.
\end{aligned}$$
We define an $\mathcal{R}$-linear operator $d$ by
$$
d(e_{ij})=\sum_{x\in {L_i}}C^{ii}_{xi}e_{xj}+C^{ij}_{ij}e_{ij}+\sum_{y\in {R_j}}C^{jj}_{jy}e_{iy}
$$
for all $e_{ij}\in \mathfrak{B}$, where the coefficients $C^{ij}_{xy}$ are subject to the following relations
$$\begin{aligned}
C^{jj}_{jk}+C^{kk}_{jk}&=0, \hspace{6pt} \text{if}\hspace{6pt} j\leqslant k;\\
C^{ij}_{ij}+C^{jk}_{jk}&=C^{ik}_{ik}, \hspace{6pt} \text{if}\hspace{6pt} i\leqslant j,j\leqslant k.
\end{aligned}$$
Then Theorem \ref{xx2.2} makes $d$ be a derivation. Then the operator $\Delta:=D-d$ is a Jordan derivation of
$I(X,\mathcal{R})$ satisfying
$$
\Delta(e_{ii})=0 \hspace{6pt}\text{and}\hspace{6pt} \Delta(e_{ij})=C^{ij}_{ji}e_{ji}
$$
for all $i\leqslant j$. Note that we have assumed $C^{xy}_{ji}=0$ if $j\nleqslant i$.
Restrict $\Delta$ to the subalgebra generated by $e_{ii},e_{jj},e_{ij},e_{ji}$ if $j\leqslant i$ which is isomorphic to
the full matrix algebra ${\rm M}_2(\mathcal{R})$.
Then \cite[Theorem 7 and Theorem 22]{JaRi} imply that $\Delta$ is a derivation and hence $C^{ij}_{ji}=0$.
This completes the proof of the theorem.
\end{proof}

\begin{remark}\label{xx3.4}
If $X$ is a finite partially ordered set (poset), then the incidence algebra $I(X,\mathcal{R})$ (also called
bigraph algebra or finite dimensional CSL algebra) is usually a triangular algebra (see \cite[P.1245]{XW}).
At this case, Theorem \ref{xx3.3} can be deduced from \cite[Theorem 3.2]{Lu} or \cite[Theorem 2.1]{ZY}.
\end{remark}


\begin{thebibliography}{}

\bibitem{Ba} K. Baclawski, {\em Automorphisms and derivations of incidence algebras},
Proc. Amer. Math. Soc., \textbf{36} (1972), 351-356.

\bibitem{Br1} M. Bre\v{s}ar, {\em Jordan derivations on semiprime rings},
Proc. Amer. Math. Soc., \textbf{104} (1988), 1003-1006.

\bibitem{Br2} M. Bre\v{s}ar, {\em Jordan mappings of semiprime rings},
J. Algebra, \textbf{127} (1989), 218-228.

\bibitem{Br3} M. Bre\v{s}ar, {\em Jordan derivations revisited},
Math. Proc. Camb. Phil. Soc., \textbf{139} (2005), 411-425.

\bibitem{CoM} S. P. Coelho and C. P. Milies, {\em Derivations of upper triangular matrix rings},
Linear Algebra Appl., \textbf{187} (1993), 263-267.

\bibitem{Cu} J. M. Cusack, {\em Jordan derivations on rings},
Proc. Amer. Math. Soc., \textbf{53} (1975), 321-324.

\bibitem{He} I. N. Herstein, {\em Jordan derivations of prime
rings}, Proc. Amer. Math. Soc., \textbf{8} (1957), 1104-1110.

\bibitem{JaRi} N. Jacobson and C. Rickart, {\em Jordan homomorphisms of rings},
Trans. Amer. Math. Soc., \textbf{69} (1950), 479-502.

\bibitem{Jo} S. J{\o}ndrup, {\em Automorphisms and derivations of upper triangular matrix rings},
Linear Algebra Appl., \textbf{221} (1995), 205-218.

\bibitem{Ko} M. Koppinen, {\em Automorphisms and higher derivations of incidence algebras},
J. Algebra., \textbf{174} (1995), 698-723.

\bibitem{Lu} F. Lu, {\em The Jordan structure of CSL algebras},
Studia Math., \textbf{190} (2009), 283-299.

\bibitem{Ma} D. Mathis, {\em Differential polynomial rings and Morita equivalence},
Comm. Algebra, \textbf{10} (1982), 2001-2017.

\bibitem{No1} A. Nowicki, {\em Derivations of special subrings of matrix rings and
regular graphs}, Tsukuba J. Math., \textbf{7} (1983), 281-297.

\bibitem{No2} A. Nowicki and I. Nowosad, {\em Local derivations of subrings of matrix rings},
Acta Math. Hungar., \textbf{105} (2004), 145-150.

\bibitem{SpDo} E. Spiegel and C. O'Donnell, {\em Incidence algebras},
Monographs and Textbooks in Pure and Applied Mathematics, vol.\textbf{206}, marcel Dekker,
New York, 1997.

\bibitem{Sp} E. Spiegel, {\em On the automorphisms of incidence algebras},
J. Algebra, \textbf{239} (2001), 615-623.

\bibitem{St} R. Stanley, {\em Structure of incidence algebras and their automorphism groups},
Bull. Amer. Math. Soc., \textbf{76} (1970), 1236-1239.

\bibitem{Wa} M. Ward, {\em Arithmetic functions on rings},
Ann. Math., \textbf{38} (1937), 725-732.

\bibitem{XW} Z.-K. Xiao and F. Wei, {\em Lie triple derivations of triangular algebras},
Linear Algebra Appl., \textbf{437} (2012), 1234-1249.


\bibitem{ZY} J.-H. Zhang and W.-Y. Yu, {\em Jordan derivations
of triangular algebras}, Linear Algebra Appl., \textbf{419} (2006),
251-255.


\end{thebibliography}
\end{document}